\documentclass[12pt]{amsart}

\usepackage{amsmath,amsthm,amssymb}
\usepackage{enumerate}
\usepackage{tikz}
\usetikzlibrary{positioning}
\usepackage{hyperref}  
\usepackage{amsfonts,amscd}
\usepackage{mathtools}
\usepackage{wasysym} 
\usepackage{graphicx}
\usepackage{psfrag}
\usepackage{datetime}
\usepackage{tikz-cd}  
\usepackage{geometry} 
\usepackage[all]{xy}

\DeclareFontFamily{U}{mathx}{\hyphenchar\font45}
\DeclareFontShape{U}{mathx}{m}{n}{
      <5> <6> <7> <8> <9> <10>
      <10.95> <12> <14.4> <17.28> <20.74> <24.88>
      mathx10
      }{}
\DeclareSymbolFont{mathx}{U}{mathx}{m}{n}
\DeclareFontSubstitution{U}{mathx}{m}{n}
\DeclareMathAccent{\widecheck}{0}{mathx}{"71}

\newtheorem{theorem}{Theorem}[section]
\newtheorem{lemma}[theorem]{Lemma}
\newtheorem{proposition}[theorem]{Proposition}

\theoremstyle{definition}
\newtheorem{definition}[theorem]{Definition}

\theoremstyle{remark}

\newcommand{\nnw}[1]{\textcolor{black}{#1}}    

\title[]{Iyama's higher Auslander correspondence via the homological theory of idempotent ideals.}
\author{Jordan McMahon}
\address{unaffiliated}
\email{jordanmcmahon37@gmail.com}
\begin{document}
\begin{abstract}
A celebrated result in representation theory is that of higher Auslander correspondence. Let $\Lambda$ an Artin algebra and $X$ a $d$-cluster-tilting module. Iyama has shown that the endomorphism ring $\Gamma$ of $X$ is a $d$-Auslander algebra, and moreover this gives a correspondence between $d$-cluster-tilting modules and $d$-Auslander algebras. We present a self-contained and concise proof using the homological theory of idempotent ideals of Auslander--Platzeck--Todorov.
\end{abstract}

\maketitle

\section{Introduction}

The aim of this note is to present a concise proof of Iyama's higher Auslander correspondence. For the benefit of the reader, it is as self-contained as possible. Let us stress that all the main arguments used are well known: Section \ref{sec1} is based on the work of Iyama \cite{iy1}, Section \ref{sec2} follows from the ideas of Auslander--Platzeck--Todorov \cite{apt}. while Section \ref{sec3} is also due to Iyama \cite{iy3}.

\section{Preliminaries} Let $K$ be a field and $\Lambda$ a finite-dimensional $K$-algebra. We denote by $\mathrm{mod}\Lambda$ the category of (finitely-generated left) $\Lambda$-modules. For morphisms $f:X\rightarrow Y$ and $g:Y\rightarrow Z$, we denote the composition by $fg$.  Let $\mathrm{add}(M)$ be the full subcategory of $\mathrm{mod}\Lambda$ composed of all $\nnw{\Lambda}$-modules isomorphic to direct summands of finite direct sums of copies of $M$. The functor $D=\mathrm{Hom}_K(-,K)$ defines a duality.

Recall that the \emph{dominant dimension} $\mathrm{dom.dim}(\Lambda)$ to be the number $n$ such that for a minimal injective resolution of $\Lambda$:
$$0\rightarrow \Lambda\rightarrow I_0\rightarrow \cdots \rightarrow I_{n-1}\rightarrow I_{n}\rightarrow \cdots$$ the modules $I_0,\ldots, I_{n-1}$ are projective-injective and $I_{n}$ is not projective. Equivalently, $\mathrm{dom.dim}(\Lambda)$ is the number $n$ such that for a minimal projective resolution of $D\Lambda$:
$$\cdots \rightarrow P_n \rightarrow P_{n-1}\rightarrow \cdots \rightarrow P_0\rightarrow D\Lambda\rightarrow 0$$ the modules $P_0,\ldots, P_{n-1}$ are projective-injective and $P_n$ is not.

\section{$d$-cluster-tilting subcategories}\label{sec1}
Let $\mathcal{C}$ be an additive subcategory of $\mathrm{mod}\Lambda$.  A \emph{$\mathcal{C}$-module} is a contravariant additive functor from $\mathcal{C}$ to the category of abelian groups. A $\mathcal{C}$-module $M$ is \emph{finitely presented} if there exists a morphism $f:X\rightarrow Y$ in $\mathrm{mod}\Lambda$ and an exact sequence 
$$\begin{tikzcd}
\mathrm{Hom}_\Lambda(C,X)\arrow{r}{\mathrm{Hom}_\Lambda(C,f)}&\mathrm{Hom}_\Lambda(C,Y)\arrow{r}&M\arrow{r}&0\end{tikzcd}$$
for all $C\in \mathcal{C}$. Denote by $\mathrm{mod}\ \mathcal{C}$ the category of finitely-presented $\mathcal{C}$-modules. If $\mathcal{C}=\mathrm{add}X$ for some $X\in\mathrm{mod}\Lambda$, it is known that the categories $\mathrm{mod}\ \mathcal{C}$ and $\mathrm{mod}\ \mathrm{End}_\Lambda(X)$ are equivalent.

A subcategory $\mathcal{C}$ of $\mathrm{mod}\Lambda$ is \emph{precovering} or \emph{contravariantly finite} if for any $M\in \mathrm{mod}\Lambda$ there is an object $C_M\in\mathcal{C}$ and a morphism $f:C_M\rightarrow M$ such that
$\mathrm{Hom}(C,-)$ is exact on the sequence 
$$\begin{tikzcd}
C_M\arrow{r}&M\arrow{r}&0
\end{tikzcd}$$
for all $C\in\mathcal{C}$. The module $C_M$ is said to be a \emph{right $\mathcal{C}$-approximation}. The dual notion of precovering is \emph{preenveloping} or \emph{covariantly finite}.  A subcategory $\mathcal{C}$ that is both precovering and preenveloping is called \emph{functorially finite}. A \emph{right $\mathcal{C}$-resolution} is a sequence
$$\begin{tikzcd}
\cdots \arrow{r}&C_1\arrow{r}&C_0\arrow{r}&M\arrow{r}&0
\end{tikzcd}$$
with $C_i\in\mathcal{C}$ for each $i$, and which becomes exact under $\mathrm{Hom}_\Lambda(C,-)$ for each $C\in\mathcal{C}$. Define a \emph{left $\mathcal{C}$-resolution} dually.

\begin{definition}\cite[Definition 2.2]{iy1}
A functorially-finite subcategory $\mathcal{C}\subseteq \mathrm{mod}\Lambda$ is a \emph{$d$-cluster-tilting subcategory} if it satisfies the following conditions:
\begin{align*}
\mathcal{C}&=\{X\in\mathrm{mod}\nnw{\Lambda}\mid \mathrm{Ext}^i_{\nnw{\Lambda}}(C,X)=0\ \forall\ 0< i<d,\  C\in\mathcal{C}\}.\\
\mathcal{C}&=\{X\in\mathrm{mod}\nnw{\Lambda} \mid \mathrm{Ext}^i_{\nnw{\Lambda}}(X,C)=0\ \forall\ 0< i<d,\  C\in\mathcal{C} \}.
\end{align*} 
If $\mathcal{C}=\mathrm{add}(M)$, then we say $M$ is a \emph{$d$-cluster-tilting module}.
\end{definition}

In particular, all projective and all injective $\Lambda$-modules are contained in $\mathcal{C}$.
\begin{theorem}\cite[Theorem 3.6.1]{iy1}\label{iyamatheorem}
Let $\mathcal{C}\subseteq \mathrm{mod} \Lambda$ be a $d$-cluster-tilting subcategory. Then 
\begin{enumerate}
\item Any $M\in\mathrm{mod}\Lambda$ has a right $\mathcal{C}$-resolution
$$\begin{tikzcd}
0\arrow{r}&C_{d-1}\arrow{r}&\cdots \arrow{r}&C_1\arrow{r}&C_0\arrow{r}&M\arrow{r}&0.
\end{tikzcd}$$
\item Any $M\in\mathrm{mod}\Lambda$ has a left $\mathcal{C}$-resolution
$$\begin{tikzcd}
0\arrow{r}&M \arrow{r}&C_0\arrow{r}&C_1\arrow{r}&\cdots\arrow{r}&C_{d-1}\arrow{r}&0.
\end{tikzcd}$$
\end{enumerate}
\end{theorem}

\begin{proof}
We give a proof for right resolutions. There is a right $\mathcal{C}$-approximation $f:C_0\rightarrow M$, since $\mathcal{C}$ is precovering. The morphism $f$ is surjective, since every projective $\Lambda$-module is in $\mathcal{C}$. Hence there is a short exact sequence 
$$\begin{tikzcd}
0 \arrow{r}&K\arrow{r}&C_0\arrow{r}{f}&M\arrow{r}&0.
\end{tikzcd}$$
For any $C\in \mathcal{C}$, there is a long exact sequence 
$$\begin{tikzcd}
0\arrow{r}&\mathrm{Hom}_\Lambda(C,K)\arrow{r}&\mathrm{Hom}_\Lambda(C,C_0) \arrow{r}{\mathrm{Hom}_\Lambda(C,f)}&\mathrm{Hom}_\Lambda(C,M)\\
\ \arrow{r}&\mathrm{Ext}^1_\Lambda(C,K)\arrow{r}&\mathrm{Ext}^1_\Lambda(C,C_0)=0.
\end{tikzcd}$$
Since $\mathrm{Hom}_\Lambda(C,f)$ is surjective, this implies $\mathrm{Ext}^1_\Lambda(C,K)=0$ and that there is an exact sequence
$$\begin{tikzcd}
0\arrow{r}&\mathrm{Hom}_\Lambda(C,K)\arrow{r}&\mathrm{Hom}_\Lambda(C,C_0) \arrow{r}{\mathrm{Hom}_\Lambda(C,f)}&\mathrm{Hom}_\Lambda(C,M)\arrow{r} &0.
\end{tikzcd}$$ 

For some $K^\prime\in \mathrm{mod}\Lambda$ there is a short exact sequence induced by the right $\mathcal{C}$-approximation $C_N\rightarrow N$:
$$\begin{tikzcd}
0 \arrow{r}&K^\prime\arrow{r}&C_N\arrow{r}&N\arrow{r}&0.
\end{tikzcd}$$Now suppose there is a $\Lambda$-module $N$ such that for all $0<i< d-1$ we have $\mathrm{Ext}^i_\Lambda(C,N)=0$ and that for any $C\in\mathcal{C}$ there is an exact sequence 
$$\begin{tikzcd}
0\arrow{r}&\mathrm{Hom}_\Lambda(C,N)\arrow{r}&\mathrm{Hom}_\Lambda(C,C_{d-2}) \arrow{r}&\cdots&\ \\ 
\ \arrow{r}&\mathrm{Hom}_\Lambda(C,C_{0}) \arrow{r}{\mathrm{Hom}_\Lambda(C,f)}& \mathrm{Hom}_\Lambda(C,M)\arrow{r} &0.
\end{tikzcd}$$ 
For any $C\in \mathcal{C}$ and any $1<i<d$, there is an exact sequence 
$$\begin{tikzcd}
 \mathrm{Ext}^{i-1}_\Lambda(C,N)\arrow{r}&\mathrm{Ext}^{i}_\Lambda(C,K^\prime)\arrow{r}&\mathrm{Ext}^{i}_\Lambda(C,C_N)=0
\end{tikzcd}$$
which, together with the above argument for $d=1$, implies $\mathrm{Ext}^{i}_\Lambda(C,K^\prime)=0$ for all $0<i<d$. So $K^\prime \in\mathcal{C}$. Additionally
there is an exact sequence
$$\begin{tikzcd}
0\arrow{r}&\mathrm{Hom}_\Lambda(C,K)\arrow{r}&\mathrm{Hom}_\Lambda(C,C_N)\arrow{r}&\mathrm{Hom}_\Lambda(C,C_{d-2})  \arrow{r}&\cdots &\\
\  \arrow{r}& \mathrm{Hom}_\Lambda(C,C_{0}) \arrow{r}&\mathrm{Hom}_\Lambda(C,M)\arrow{r} &0.
\end{tikzcd}$$ This finishes the proof. 
\end{proof}

\section{Homological theory of idempotent ideals}\label{sec2}
Let $\Gamma$ be a finite-dimensional algebra, and $P$ be a projective $\Gamma$-module. Let $\mathcal{I}=\tau_P(\Gamma)$, the trace of $P$ in $\Gamma$ which is the ideal generated by the homomorphic images of $P$ in $\Gamma$, be the idempotent ideal corresponding to $P$. When $P=\Gamma e$, then $\mathcal{I}=\langle e\rangle$, the two-sided ideal generated by the idempotent $e$. For a positive integer $d$, we define $\mathbf{P}_{d-1}$ to be the full subcategory of $\mathrm{mod}\Gamma$ consisting of the $\Gamma$-modules $X$ having a projective resolution
$$\cdots \rightarrow P_1\rightarrow P_0\rightarrow X\rightarrow 0$$ with $P_i\in\mathrm{add}(P)$ for all $0\leq i\leq d-1$. There is a characterisation of $\mathbf{P}_{d-1}$.
\begin{proposition}\cite[Proposition 2.4]{apt}\label{apti}
Let $P$ be a projective $\Lambda$-module. The following conditions are equivalent for a $\Gamma$-module $X$:
\begin{enumerate}[(i)]
\item $X$ is in $\mathbf{P}_{d-1}$.
\item $\mathrm{Ext}^i_\Gamma(X,Y)=0$ for all $A/\mathcal{I} $-modules $Y$ and $0<i<d$.
\item $\mathrm{Ext}^i_\Gamma(X,Y)=0$ for all injective $A/\mathcal{I} $-modules $Y$ and $0<i<d$.
\end{enumerate}
\end{proposition}
\begin{proof}
by induction on $d$.
\end{proof}
We omit a more detailed proof of Proposition \ref{apti}, since it is not used in the sequel. Note that in the special case $X=\Gamma$ (as used by Iyama \cite[Lemma 3.5.1]{iy3}) Proposition \ref{apti} reduces the proof of the following result, which also uses arguments from Sections 5 and 6 of \cite{auslander}.
\begin{theorem}\cite[Theorem 3.2]{apt}\label{aptii}
Let $P$ be a projective $\Gamma$-module and $\Lambda:=\mathrm{End}_\Gamma(P)$. For any $Y\in\mathrm{mod}\Gamma$ and any $d\geq 1$, the functor $G:=\mathrm{Hom}_\Gamma(P,-)$ induces an isomorphism $$\mathrm{Ext}_\Gamma^{d-1}(X,Y)\rightarrow \mathrm{Ext}^{d-1}_\Lambda(GX,GY)$$ provided $X\in \mathbf{P}_d$. 
 \end{theorem} 

\begin{proof}We prove by induction, omitting the case $i=2$ since the argument requires only minor modifications. Let $Y$ be a fixed, but arbitrary $\Gamma$-module. Recall that there is a canonical isomorphism $\mathrm{Hom}_\Gamma(P,Y)\cong \mathrm{Hom}_\Lambda(GP,GY)$. First suppose $X\in\mathbf{P}_1$; so there is an exact sequence $P_1\rightarrow P_0\rightarrow X\rightarrow 0$ such that $P_0,P_1\in\mathrm{add}P$. Then there is a commutative diagram

$$\begin{tikzcd}
0\arrow{r}& \mathrm{Hom}_\Gamma(X,Y)\arrow{r}\arrow{d}{}& \mathrm{Hom}_\Gamma(P_0,Y)\arrow{d}{\simeq}\arrow{r}&  \mathrm{Hom}_\Gamma(P_1,Y)\arrow{d}{\simeq}\\
0\arrow{r}& \mathrm{Hom}_\Lambda(GX,GY)\arrow{r}& \mathrm{Hom}_\Lambda(GP_0,GY)\arrow{r}&  \mathrm{Hom}_\Lambda(GP_1,GY)\\
\end{tikzcd}$$
implying $\mathrm{Hom}_\Gamma(X,Y)\cong \mathrm{Hom}_\Lambda(GX,GY)$ by the Five Lemma. 
 Now suppose for some $i\geq 2$ that for all $M\in \mathbf{P}_{i}$ there are isomorphisms $$\mathrm{Ext}_\Gamma^{i-1}(M,Y)\rightarrow \mathrm{Ext}^{i-1}_\Lambda(GM,GY).$$ Let $X\in \mathbf{P}_{i+1}$ and consider an exact sequence $0\rightarrow K\rightarrow P_0\rightarrow X\rightarrow 0$ with $P_0\in \mathrm{add}(P)$. 
Since $K\in \mathbf{P}_{i}$, we get $\mathrm{Ext}^{i-1}_\Gamma(K,Y)\cong \mathrm{Ext}^{i-1}_\Lambda(GK,GY)$. From the commutative diagram
$$\begin{tikzcd}
0\arrow{r}& \mathrm{Ext}^{i-1}_\Gamma(K,Y)\arrow{r}\arrow{d}& \mathrm{Ext}^{i}_\Gamma(X,Y)\arrow{d}\arrow{r}& 0\\
 0\arrow{r} &\mathrm{Ext}^{i-1}_\Lambda(GK,GY)\arrow{r}& \mathrm{Ext}^{i}_\Lambda(GX,GY)\arrow{r} &0
\end{tikzcd}$$
we find also $\mathrm{Ext}^{i}_\Gamma(X,Y)\cong  \mathrm{Ext}^{i}_\Lambda(GX,GY)$. 
\end{proof}
\section{Higher Auslander correspondence}\label{sec3}
Recall for a $\Lambda$-module $X$ such that $\Lambda,D\Lambda\in \mathrm{add}(X)$, the endomorphism algebra $\Gamma:=\mathrm{End}_\Lambda(X)$ has projective modules given by $\mathrm{add}(\mathrm{Hom}_\Lambda(X,X))$,  and projective-injective modules given by $\mathrm{add}( \mathrm{Hom}_\Lambda(X,D\Lambda))$.
\begin{proposition}\label{iyi}
Suppose $\Lambda$ is an artin algebra and $X$ is a $d$-cluster-tilting module in $\mathrm{mod}\Lambda$. Then $\Gamma:=\mathrm{End}_\Lambda(X) $ satisfies $$\mathrm{dom.dim}(\Gamma)\geq d+1\geq\mathrm{gl.dim} (\Gamma),$$ i.e. $\Gamma$ is a \emph{$d$-Auslander algebra}.
\end{proposition}

\begin{proof}
Let $F:=\mathrm{Hom}_\Lambda(X,-)$ and $\mathcal{C}:=\mathrm{add}(X)$.
We have that $_\Gamma\Gamma=FX$. Suppose $X$ has an injective resolution:
$$0\rightarrow X\rightarrow I_0\rightarrow I_1\rightarrow \cdots$$
Since $X$ is $d$-cluster tilting, we have an exact sequence
$$0\rightarrow FX\rightarrow FI_0 \rightarrow \cdots \rightarrow FI_{d}$$
whereby $FI_j$ is projective-injective for each $0\leq j\leq d$. Hence  $\mathrm{dom.dim}(\Gamma )\geq d+1$.
For each $M\in\mathrm{mod}\ \Gamma$, we show $\mathrm{proj.dim}(M)\leq d+1$, and hence $\mathrm{gl.dim}(\Gamma)\leq d+1$. Since $\mathrm{mod}\ \Gamma$ is equivalent to $\mathrm{mod}\ \mathcal{C}$, there exist $C_{-2},C_{-1}\in\mathcal{C}$ and an exact sequence:
$$FC_{-1} \rightarrow FC_{-2}\rightarrow M\rightarrow 0.$$ By Theorem \ref{iyamatheorem}, this extends to an exact sequence
$$0\rightarrow FC_d\rightarrow \cdots \rightarrow FC_0\rightarrow FC_{-1}\rightarrow FC_{-2} \rightarrow M\rightarrow 0$$ such that $C_i\in\mathcal{C}$ for all $-2\leq i\leq d$. Hence $d+1\geq\mathrm{gl.dim}( \Gamma)$ and we are done.
\end{proof}

\begin{proposition}\label{iyii}
Let $\Gamma$ be a $d$-Auslander algebra, let $P$ be a minimal projective-injective generator and let $\Lambda:=\mathrm{End}_\Gamma(P)$. Then there is a $d$-cluster-tilting subcategory $\mathcal{C}\subseteq\mathrm{mod}\Lambda$.
\end{proposition}

\begin{proof}
Let $\Lambda:=\mathrm{End}_\Gamma(P)$ and $G:\mathrm{Hom}_\Gamma(P,-)$. We will show $\mathcal{C}:=\mathrm{add}(G_\Gamma\Gamma)\subseteq \mathrm{mod}(\Lambda)$ is $d$-cluster tilting. By assumption $D\Gamma \in \mathbf{P}_{d}$. So Theorem \ref{aptii} implies an isomorphism $$\mathrm{Ext}_\Gamma ^i(X,Y)\rightarrow \mathrm{Ext}^i_\Lambda(GX,GY)$$ for all $X,Y\in\mathrm{add}(_\Gamma\Gamma)$ and all $0<i<d$. Hence $\mathrm{Ext}^i_\Lambda(GX,GY)=0$ for all $0<i<d$. To show $\mathrm{add}(G_\Gamma\Gamma)$ is a $d$-cluster-tilting subcategory, we have to show maximality. So suppose on the other hand that there exists some $M\notin\mathcal{C}$ such that $\mathrm{Ext}^i_A(C,M)=0$ for all $0<i<d$ and all 
$C\in\mathcal{C}$. Let $F:=\mathrm{Hom}_\Lambda(G_\Gamma \Gamma, -)$ and suppose $M$ has an injective resolution:
$$0\rightarrow M\rightarrow I_0\rightarrow I_1\rightarrow \cdots.$$
Since $\mathrm{Ext}^i_A(C,M)=0$ for all $0<i<d$, we have an exact sequence
$$0\rightarrow FM\rightarrow FI_0 \rightarrow \cdots \rightarrow FI_{d}$$
whereby $FI_j$ is projective-injective for each $0\leq j\leq d$. Let $N$ be the $\Lambda$-module $N:=\mathrm{coker}(FI_{d-1}\rightarrow FI_d)$. Since $\mathrm{gl.dim}(\Gamma)\leq d+1$, the sequence 
$$0\rightarrow FM\rightarrow FI_0 \rightarrow \cdots \rightarrow FI_{d}\rightarrow N\rightarrow 0$$ is a projective resolution of $N$, and hence $FM$ is projective, in other words $M\in\mathcal{C}.$ Dually, for any $M\notin\mathcal{C}$ such that $\mathrm{Ext}^i_A(M,X)=0$ for all $0<i<d$ and all 
$C\in\mathcal{C}$, we must have $M\in\mathcal{C}$. Therefore $\mathcal{C}\subseteq\mathrm{mod} \Lambda$ is $d$-cluster  tilting.
\end{proof}

\begin{lemma}\label{iyamalemma}
Let $\Lambda,\Gamma,X$ as above. There are mutually inverse equivalences $$F:\mathrm{add}(X)\rightarrow \mathrm{add}(_\Gamma \Gamma)$$ $$G:\mathrm{add}(_\Gamma \Gamma)\rightarrow \mathrm{add}(X)$$
\end{lemma}

\begin{proof}
That $F$ is an equivalence has already been discussed. Moreover \begin{align*}G\circ F&=\mathrm{Hom}_\Gamma(P, \mathrm{Hom}_\Lambda(X,-))\\ &=\mathrm{Hom}_\Lambda(X\otimes_\Gamma P, -)\\&=\mathrm{Hom}_\Lambda(X\otimes_\Gamma \mathrm{Hom}_\Gamma(X,\Lambda),-)\\&=\mathrm{Hom}_\Lambda(\Lambda,-)\\&=1
\end {align*}
\end{proof}
Let two $d$-cluster-tilting modules $M,N\in\mathrm{mod}\Lambda$ be \emph{equivalent} whenever the categories $\mathrm{add}(M)$ and $\mathrm{add}(N)$ are equivalent. 
\begin{theorem}[$d$-Auslander correspondence]\cite[Theorem 0.2]{iy3}
For any $d\geq 1$ there exists a bijection between the equivalence classes of $d$-cluster-tilting modules $X\in\mathrm{mod}\Lambda$ and the set of Morita-equivalence classes of $d$-Auslander algebras, given by $X\mapsto \mathrm{End}_\Lambda(X)$. 
\end{theorem}

\begin{proof}
This follows from Proposition \ref{iyi}, Proposition \ref{iyii} and Lemma \ref{iyamalemma}\end{proof}

\bibliographystyle{amsplain}
\bibliography{extendinging}

\end{document}